\documentclass[a4paper, 12pt, oneside]{amsart}
\usepackage[english]{babel}
\usepackage{amsfonts}
\usepackage{amssymb}
\usepackage[leqno]{amsmath}
\usepackage{amsthm}
\usepackage{amscd}
\usepackage{xypic}
\usepackage[curve]{xy}
\usepackage{epsfig}
\usepackage{enumerate}

\input xy
\xyoption{all}
\newtheorem{theorem}{Theorem}[section]
\newtheorem{proposition}[theorem]{Proposition}

\newtheorem{corollary}[theorem]{Corollary}

\newtheorem{example}[theorem]{Example}
\newtheorem{lemma}[theorem]{Lemma}
\newtheorem{remark}[theorem]{Remark}

\def\Rad{\mathrm{Rad}}

 \DeclareMathOperator{\Hom}{Hom}

\def\Cohom{\mathrm{Cohom}}

\def\Ext{\mathrm{Ext}}
\def\Hom{\mathrm{Hom}}
\def\Irr{\mathrm{Irr}}
\def\inj{\mathrm{inj}}

\def\rad{\mathrm{rad}}

\def\dim{\mathrm{dim}}

\def\End{\mathrm{End}}
\def\Soc{\mathrm{soc\hspace{0.1cm}}}

\def\Ker{\mathrm{Ker\hspace{0.1cm}}}
\def\Ima{\mathrm{Im\hspace{0.1cm}}}

\def\T{\mathcal{T}}

\def\M{\mathcal{M}}

\def\point#1{*+[o][F-]{\text{\scriptsize $#1$}}}
\def\isolated#1{\xymatrix{\point{#1}}}

\title{The valued Gabriel quiver of a wedge product and semiprime coalgebras}
\author{Gabriel Navarro}
\address{Department of Computer Sciences and A.I\\
University of Granada \\
c. El Greco s/n\\ E-51002\\ Ceuta\\ Spain}
\email{gnavarro@ugr.es}
\thanks{Research supported by Spanish MEC project MTM2007-66666, and TIC-111
 (Junta de Andaluc{\'\i}a Research Group).}

\begin{document}

\maketitle

\begin{abstract}
We make a first approach to the representation theory of the wedge product of
coalgebras by means of the description of its valued Gabriel quiver. Then we define semiprime
coalgebras and study its category of comodules by the use of localization techniques.
In particular,
we prove that, whether its Gabriel quiver is locally finite, any monomial semiprime fc-tame coalgebra is string.
 We also prove a weaker version of Eisenbud-Griffith theorem for coalgebras, namely, any hereditary semiprime strictly
quasi-finite coalgebra is serial.
\end{abstract}

\section{Introduction}

Throughout this paper we fix a field $k$ and denote by $C$ a $k$-coalgebra.
Unless otherwise stated, we shall assume that $C$ is basic \cite{chin}\cite{simson1},
 that is, $C$ and its
(left or right) socle have decompositions

\begin{equation}\label{decomp1}
C=\bigoplus_{i\in I_C} \hspace{0.1cm} E_i \hspace{0.5cm} \text{ and
} \hspace{0.5cm} \Soc C=\bigoplus_{i\in I_C} \hspace{0.1cm} S_i
\hspace{0.1cm},
\end{equation}
 where $\{E_i\}_{i\in
I_C}$ is a complete set of pairwise non-isomorphic
\emph{indecomposable injective} right $C$-comodules and
$\{S_i\}_{i\in I_C}$ is a complete set of pairwise non-isomorphic
\emph{simple} right (and left) $C$-comodules.

Given two subcoalgebras $A$ and $B$ of $C$, the \emph{wedge product} \cite{sweedler}
of $A$ and $B$ in $C$,
$A\wedge^C B=\Ker(\xymatrix{C\ar[r]^-{\Delta}& C\otimes
C \ar[r]^-{pr\otimes pr}& C/A\otimes C/B})$,
 where $\Delta$ is the comultiplication
of $C$ and $pr$ is the standard projection. Equivalently, $A\wedge^C B$ is the set of
elements $x\in C$ such that $\Delta(x)\in A\otimes C+C\otimes B$, or equivalently,
$A\wedge^C B=(A^{\perp}B^{\perp})^{\perp}$, where $\perp$ denotes the standard orthogonality. If there is no ambiguity, we simply denote
the wedge product by $A\wedge B$.

In the ``coalgebraic'' setting, prime coalgebras,
as those which cannot been decomposed non-trivially
as a wedge product, has been investigated by several authors.
The notion appeared first in Takeuchi's PhD Thesis
 \cite{takeuchithesis} under cocommutative conditions.
Nevertheless, as it was pointed out in \cite{jmnr},
it could be defined in the same way without this restriction.
 This is done in \cite{nek}, where the authors analyze the
Zariski topology attached to the set of prime subcoalgebras
 over a field. In \cite{jmnr}, pointed prime coalgebras over
 a field are studied from the graphical point of view of path
 coalgebras of a quiver. In a more general setting, it has been developed in the
 context of coalgebras over a commutative ring \cite{WijayantiWisbauer}
under the name of wedge coprime coalgebras;  corings \cite{abu} \cite{abu2},
calling them fully coprime corings; or modules and comodules \cite{abu}
\cite{ferrero}; suffering progressive generalizations which,
whether the base ring is a field, coincide with the notion
given by Takeuchi. Nevertheless, unless in a very tangential way,
 none of the above papers deals with the category
 of comodules of a prime coalgebra or coring.

In this paper we make a first approach
 to the representation theory of a wedge product
 of arbitrary coalgebras over a field
 and describe its right valued
 Gabriel quiver \cite{justus} (Theorem \ref{strongly}).
 From this point of view, we deal with semiprime
 coalgebras
  proving that its valued Gabriel quiver should have
an specific shape (Theorem \ref{main}), as a generalization of a result
obtained in \cite{jmnr}. As a consequence, we give a weaker version of the coalgebraic analog of
 a theorem of Eisenbud and Griffith \cite{eisenbud} proven in \cite{GTN} (Corollary \ref{eisenbud}). In the last section, we shall apply these results in order to get certain properties
of the category of comodules of a semiprime coalgebra. In particular, we highlight that, over an
algebraically closed field, a monomial semiprime fc-tame coalgebra whose Gabriel quiver is locally finite is string in the sense of \cite{simson2} (Theorem \ref{biserial}). We would like to remark that, from this perspective, semiprime coalgebras seems to be a more
 appropriated class than prime
 coalgebras since this class is closed under direct sums.

We also would like to remark that we follow
 the nomenclature of \cite{jmnr} and call the coalgebras
``prime'' or ``semiprime'', whilst the papers above-mentioned
 make use of the word ``coprime'' or ``cosemiprime''.
 We think that add the prefix ``co'' in  ``coalgebra''
is enough for pointing out the dual nature of this notion
 and makes more readable the manuscript.
The reader also should note that prime coalgebras over a field
as described in \cite{ferrero} differ from the ones treated here, since
all of them are simple.

All along the paper
 we shall make use of the localization techniques develop
 in \cite{jmn2}, \cite{jmnr2}, \cite{navarro},
\cite{simson06} or \cite{simson07} which have been
 showed to be an efficient tool for developing the
 Representation Theory of Coalgebras. Actually, we
 partially solve a subtle mistake in the proof of
 \cite[Theorem 4.2]{jmnr}. Therefore, for the convenience of the reader, let us remind
 the localization theory in category of comodules.
 Throughout we denote by $\M^C_f$, $\M^C_{qf}$ and $\M^C$ the
category of finite dimensional, quasi-finite and all right
$C$-comodules, respectively.

Let $\T$ be a \emph{dense} subcategory (or a \emph{Serre class}) of the category $\M^C$, $\T$ is said to be \emph{localizing} (cf. \cite{gabriel}) if the quotient
functor $T:\M^C\rightarrow \M^C/\T$ has a right adjoint functor $S$,
called the \emph{section functor}.  If the section functor is exact,
$\T$ is called \emph{perfect localizing}.

From the general theory of localization in Grothendieck categories
\cite{gabriel}, it is well-known that there exists a one-to-one
correspondence between localizing subcategories of $\M^C$ and sets
of indecomposable injective right $C$-comodules, and, as a
consequence, sets of simple right $C$-comodules. More precisely, a
localizing subcategory is determined by an injective right
$C$-comodule $E=\oplus_{j\in J} E_j$, where $J\subseteq I_C$
(therefore the associated set of indecomposable injective comodules
is $\{E_j\}_{j\in J}$). Then $\M^C/\T\simeq \M^D$, where $D$ is the
coalgebra of coendomorphism $\Cohom_C(E,E)$ (cf. \cite{takeuchi} for
definitions), and the quotient and section functors are
$\Cohom_C(E,-)$ and $-\square_D E$, respectively.

In \cite{cuadra}, \cite{jmnr} and \cite{woodcock}, localizing
subcategories are described by means of idempotents in the dual
algebra $C^*$. In particular, it is proved that the quotient
category $\M^C/\T$ is the category of right comodules over the
coalgebra $eCe$, where $e\in C^*$ is an idempotent associated to the
localizing subcategory $\T$ (that is, $E=Ce$, where $E$ is the
injective right $C$-comodule associated to the localizing
subcategory $\T$). The coalgebra structure of $eCe$ (cf.
\cite{radford}) is given by
$$\Delta_{eCe} (exe)=\displaystyle \sum_{(x)} ex_{(1)}e\otimes
ex_{(2)}e \hspace{0.4cm}\text{ and }\hspace{0.4cm}
\epsilon_{eCe}(exe)=\epsilon_C(x)$$ for any $x\in C$, where
$\Delta_C(x)= \sum_{(x)} x_{(1)} \otimes x_{(2)}$ using the
sigma-notation of \cite{sweedler}. For
completeness, we recall from \cite{cuadra} (see also \cite{jmnr})
the following description of the localizing functors. We recall
that, given an idempotent $e\in C^*$, for each right $C$-comodule
$M$, the vector space $eM$ is endowed with a structure of right
$eCe$-comodule given by
$$\rho_{eM}(ex)=\sum_{(x)} ex_{(1)}\otimes
ex_{(0)} e$$ where $\rho_{M}(x)=\sum_{(x)} x_{(1)}\otimes x_{(0)}$
using the sigma-notation of \cite{sweedler}.

 The localization in categories of comodules over path
coalgebras is described in detail in \cite{jmn2} and \cite{jmnr}.
Briefly, following the notation of these papers, let $Q=(Q_0,Q_1)$ be a (possibly infinite) quiver. The path algebra $kQ$ can be endowed with
a coalgebra structure with comultiplication given by
$$\Delta(p)=e_j\otimes p+p\otimes e_i+\sum_{i=1}^{m-1}\alpha_m\cdots
 \alpha_{i+1}\otimes
\alpha_i\cdots \alpha_1=  \sum_{\eta \tau=p} \eta\otimes\tau$$
for any path $p=\alpha_m\cdots \alpha_1$ in $Q$ from $e_i$ to $e_j$, and
for a trivial path, $e_i$, $\Delta(e_i)=e_i\otimes e_i$.

Given $X\subseteq Q_0$, a path $p=\alpha_n\cdots\alpha_1$ in $Q$ is said to
be a \emph{cell} in $Q$ relative to $X$ (shortly a cell) if
$s(\alpha_1),t(\alpha_m)\in X$ and $s(\alpha_i)\notin X$ for all $i=2,\ldots
,n-1$, where $s(\alpha)$ and $t(\alpha)$ denote the source and the sink of an arrow or
a path $\alpha$.

The localizing subcategories of $\M^{kQ}$ are in one-to-one correspondence with subsets
of vertices of $Q$ and also, in one-to-one correspondence with idempotents of the dual
algebra $(kQ)^*$. Then, given $X_e\subseteq Q_0$ corresponding to an idempotent $e\in (kQ)^*$, $e(kQ)e\cong kQ^e$, where $(Q^e)_0=X_e$ and the arrows in $Q^e$ from a vertex $x$
to a vertex $y$ is the number of different cells relative to $X_e$ between these vertices, see \cite{jmn2}.

\section{The valued Gabriel quiver of a wedge product}

The valued Gabriel quiver of a coalgebra carries the information of the second piece of its coradical filtration. Therefore, it used to be one of the first invariant under consideration when dealing with its representation theory. Following \cite{justus}, let us
recall that the right \emph{valued Gabriel quiver} $(Q_C,d_C)$
of a basic coalgebra $C$ as follows: the set of vertices of $(Q_C,d_C)$
is the set of simple right $C$-comodules $\{S_i\}_{i\in I_C}$, and
there exists a unique valued arrow
$$\xymatrix{ S_i \ar[rr]^-{(d'_{ij},d''_{ij})} & & S_j}$$ if and
only if $\Ext_C^1(S_i,S_j)\neq 0$ and,
$$\text{$d'_{ij}=\dim_{G_i} \Ext_C^1(S_i,S_j)$   and
$d''_{ij}=\dim_{G_j} \Ext_C^1(S_i,S_j)$},$$ as a right $G_i$-module
and as a left $G_j$-module, respectively.
When the values of the arrows are irrelevant for our purposes, we shall denote the valued Gabriel quiver
of $C$ simply by $Q_C$.

In \cite{simson06}, the valued Gabriel quiver of $C$ is described
through the notion of irreducible morphisms between indecomposable
injective right $C$-comodules. Let us denote by $\inj^C$ (respect.
${^C}\inj$) the full subcategory of $\M^C$ (respect. ${^C}\M$)
formed by socle-finite (i.e., comodules whose socle is
finite-dimensional) injective right (respect. left) $C$-comodules.
Let $E$ and $E'$ be two comodules in $\inj^C$. A morphism
$f:E\rightarrow E'$ is said to be irreducible if $f$ is not an
isomorphism and given a factorization
$$\xymatrix{ E \ar[rr]^-{f} \ar[rd]_-{g} &    &   E'\\
&  Z   \ar[ru]_{h} &  }$$ of $f$, where $Z$ is in $\inj^C$, $g$ is
a section, or $h$ is a retraction. Analogously to the case of
finite-dimensional algebras, there it is proven that the set of
irreducible morphism $\Irr_C(E_i,E_j)$ between two indecomposable
injective right $C$-comodules $E_i$ and $E_j$ is isomorphic, as
$G_j$-$G_i$-bimodule, to the quotient
$\rad_C(E_i,E_j)/\rad_C^2(E_i,E_j)$. We recall that, for each two
indecomposable injective right $C$-comodules $E_i$ and $E_j$, the
\emph{radical} of $\Hom_C(E_i,E_j)$ is the $K$-subspace
$\rad_C(E_i,E_j)$ of $\Hom_C(E_i,E_j)$ generated by all
non-isomorphisms. Observe that if $i\neq j$, then
$\rad_C(E_i,E_j)=\Hom_C(E_i,E_j)$. The square of $\rad_C(E_i,E_j)$
is defined to be the $K$-subspace
$$\rad_C^2(E_i,E_j) \subseteq \rad_C(E_i,E_j) \subseteq \Hom_C(E_i,E_j)$$
generated by all composite homomorphisms of the form
$$\xymatrix{E_i \ar[r]^-{f} & E_k \ar[r]^-{g} & E_j,}$$ where
$f\in \rad_C(E_i,E_k)$ and $g\in \rad_C(E_k,E_j)$. The $m$th power
$\rad_C^m(E_i,E_j)$ of $\rad_C(E_i,E_j)$ is defined analogously, for
each $m>2$.

Let us now prove a generalization of \cite[Theorem 1.7]{montgomery2} which describes the right valued Gabriel quiver by means of the wedge product of simple right comodules.

\begin{proposition}\label{Gabrielquiver}
Let $C$ be a coalgebra and, $S_i$ and $S_j$ be two simple right $C$-comodules. There exists an arrow in $(Q_C,d_C)$ from $S_j$ to $S_i$ if and only if $(S_i\wedge^C S_j)/(S_i+S_j)\neq 0$. In such a case, the arrow is labeled by $d_{ji}=(d^1_{ji},d^2_{ji})$, where $d^1_{ji}=\dim_{G_j}(S_i\wedge^C S_j)/(S_i+S_j)$ as right $G_j$-comodule, and $d^2_{ji}=\dim_{G_i}(S_i\wedge^C S_j)/(S_i+S_j)$
 as left $G_i$-comodule.
\end{proposition}
\begin{proof}
Following \cite[Lemma 3.9]{cuadra2}, for any right $C$-comodule $I$ and any subcoalgebra $E$, there exists an isomorphism $(I\wedge^C E)/I\cong C/I\square_C E$ of right $E$-comodules. Therefore,
$$\frac{S_i\wedge^C S_j}{S_i}\cong C/S_i\square_C S_j\cong (\oplus_{k\neq i}E_k\square_C S_j)\oplus E_i/S_i\square_C S_j$$
In general, for any right $C$-comodule $M$, $M\square_C S_j$ is the direct sum of the simple right $C$-comodules isomorphic to $S_j$ appearing in the socle of $M$. Therefore
$\oplus_{k\neq i}E_k\square_C S_j\cong S_j$ if $j\neq i$ and zero otherwise. Furthermore,
$E_i/S_i\square_C S_j\cong S_j^{(r)}$, where $$r=\dim_{G_j} \Hom_C(S_j,E_i/S_i)=\dim_{G_j}\Ext^C(S_j,S_i),$$ see \cite{navarro}. Now,
 $$\frac{S_i\wedge S_j}{S_i+S_j}\cong \frac{(S_i\wedge S_j)/S_i}{(S_i+S_j)/S_i}\cong
 S_j^{(r)}$$ as right $S_j$-comodules. Then, the equivalence and the calculation of the first component of the label hold.

We recall from \cite{GTN} or \cite{justus2} that the right valued Gabriel quiver and the left valued Gabriel quiver of $C$ are opposite one to each other. Now, it is enough to apply the left version of the formula of \cite[Lemma 3.9]{cuadra2} and a similar reasoning as above in order to calculate $d$ and prove the statement.
\end{proof}

This can be generalized taking into account the notion of predecessor defined in \cite{navarro}. We remind that, given a simple $C$-comodule $S_i$, we say that a simple $C$-comodule $S_j$ is an
$n$-\emph{predecessor} of $S_i$ if $\Ext^1_C(S_j, \Soc^nE_i)\neq
0$ for some $n>0$, or equivalently, if
$S_j\subseteq \Soc(E_i/\Soc^{n}E_i)\cong \Soc^{n+1}E_i/\Soc^{n} E_i$ for some $n>0$, where $E_i$ is the injective envelope of $S_i$.

\begin{proposition}
Let $C$ be a coalgebra and, $S_i$ and $S_j$ be two simple right $C$-comodules. $S_j$ is a $n$-predecessor of $S_i$ if and only if $(\Soc^{n} S_i\wedge^C S_j)/(\Soc^n S_i+S_j)\neq 0$. In such a case, the number of indecomposable direct summands of $\Soc^{n+1}E_i/\Soc^{n} E_i$ isomorphic to $S_j$ is $\dim_{G_j} (\Soc^{n} S_i\wedge^C S_j)/(\Soc^2 S_i+S_j)$ as right $G_j$-comodule.
\end{proposition}
\begin{proof}
It is similar to the former proof. Simply consider that
$$\frac{\Soc^nS_i\wedge^C S_j}{\Soc^nS_i}\cong C/\Soc^nS_i\square_C S_j\cong (\oplus_{k\neq i}E_k\square_C S_j)\oplus E_i/\Soc^nS_i\square_C S_j$$ and $\Soc (E_i/\Soc^nS_i)=\Soc^{n+1}E_i/\Soc^{n}E_i$.
\end{proof}

\begin{proposition}\label{simplearrows} Let $C$ be a coalgebra and, $A$ and $B$ two
subcoalgebras of $C$. Then the following assertions hold:
\begin{enumerate}[$(a)$]
\item Let $\mathfrak{G}_A=\{S_j\}_{j\in I_A}$, $\mathfrak{G}_B=\{S_k\}_{k\in I_B}$ and $\mathfrak{G}_{A\wedge^C B}=\{S_i\}_{i\in I_{A\wedge B}}$ be a complete set of
pairwise non-isomorphic simple right $A$-comodules,
$B$-comodules and $A\wedge^C B$-comodules, respectively. Then $\mathfrak{G}_{A\wedge^C B}=\mathfrak{G}_A\cup \mathfrak{G}_B$.
\item Given two simple $A\wedge^C B$-comodules $S$ and $T$:
$$S\wedge^{A\wedge B}T=\left\{
                        \begin{array}{ll}
                          S\oplus T & \text{if $S\nsubseteq A$ and $T\nsubseteq B$} \\
                          S\wedge^C T & \text{if $S\subseteq A$ and $T\subseteq B$} \\
                           S\wedge^A T & \text{if $S\subseteq A$ and $T\nsubseteq B$}\\
                           S\wedge^B T & \text{if $S\nsubseteq A$ and $T\subseteq B$}
                        \end{array}
                      \right.$$
\end{enumerate}
\end{proposition}
\begin{proof}
\begin{enumerate}[$(a)$]
\item
It is obvious that $\mathfrak{G}_A\cup \mathfrak{G}_B \subseteq \mathfrak{G}_{A\wedge^C B}$ since $A$ and $B$ are subcoalgebras of $A\wedge^C B$.
Let now $S$ be a simple right $A\wedge^C B$-comodule.
Then $\Delta(S)\subseteq S\otimes (A\wedge^C B)$ and, since $S\subseteq A\wedge^C B$, $\Delta(S)\subseteq A\otimes C+C\otimes B$. Now, if $S\subseteq A$, $S$ is a $A$-comodule. If not, $S\cap A=\emptyset$ so $\Delta(S) \subseteq S\otimes B$, and hence $S$ is a $B$-comodule.
\item If $S\subseteq A$ and $T\subseteq B$, $S\wedge^C T\subseteq A\wedge^C B$, and therefore $$S\wedge^{A\wedge B} T=(S\wedge^C T)\cap(A\wedge^C B)=S \wedge^C T.$$

    Assume that $S\nsubseteq A$ and $T\nsubseteq B$ and let $0\rightarrow S \rightarrow M \rightarrow T \rightarrow 0$ be a non-split short exact sequence of $A\wedge^C B$-comodules. Then $M$ is a serial $A\wedge^C B$-comodule with simple socle $A\subseteq B$. In particular, $A\cap M=0$. Since $M\subset A\wedge^C B$ and $\Delta(M)\subseteq M\otimes C$, it follows that $\Delta(M)\subseteq M\otimes B$ and so $M$ is a right $B$-comodule. Hence $S$ is a right $B$-comodule and we get a contradiction. Thus $\Ext^{A\wedge B}(T,S)=0$. In particular $S\wedge^{A\wedge B}T=S\oplus T$.

Let us suppose that $S\subseteq A$ and $T\nsubseteq B$, and $x\in S\wedge^{A\wedge B} T$. Then we may write $\Delta(x)=\sum x_i\otimes y_i+\sum x_j\otimes y_j$, where $x_i\in S$, $y_j\in T$ and $y_i,x_j\in A\wedge B$. Nevertheless $x\in A\wedge^C B$ and so $\Delta(x)\in A\otimes C+C\otimes B$. Since $T\nsubseteq B$, we deduce that the elements $x_j\in A$. Thus $\Delta(x)\in A\otimes (A\wedge^C B)$ and $x=(id \otimes \epsilon)\Delta (x) \in A$. Hence $x\in S\wedge^A T$. Analogously, if $S\nsubseteq A$ and $T\subseteq B$, $S\wedge^{A\wedge B}T=S\wedge^B T$.
\end{enumerate}
\end{proof}

\begin{theorem}\label{strongly}
 Let $C$ be a coalgebra and $Q_C$ its valued Gabriel quiver. Let $A$ and $B$ two
subcoalgebras of $C$ whose valued Gabriel quiver are $Q_A$ and $Q_B$, respectively. Then the valued Gabriel quiver of $A\wedge^C B$ is described as follows:
\begin{enumerate}
\item The set of vertices of $Q_{A\wedge B}$ is the union of the set of vertices of $Q_A$ and $Q_B$, both viewed as valued subquivers of $Q_C$.
\item Given two simple $A\wedge^C B$-comodules $S$ and $T$.
\begin{enumerate}[i)]
\item If $S\nsubseteq A$ and $T\nsubseteq  B$, there is no an arrow  in $Q_{A\wedge^C B}$ from $T$ to $S$.
\item If $S\subseteq A$ and $T\subseteq B$, there is an arrow $\xymatrix{ T \ar[r]^-{(c,d)} & S}$ in $Q_{A\wedge^C B}$ if and only if it is so in $Q_C$.
\item If $S\subseteq A$ and $T\nsubseteq B$, there is an arrow $\xymatrix{ T \ar[r]^-{(c,d)} & S}$ in $Q_{A\wedge^C B}$ if and only if it is so in $Q_A$.
\item If $S\nsubseteq A$ and $T\subseteq B$, there is an arrow $\xymatrix{ T \ar[r]^-{(c,d)} & S}$ in $Q_{A\wedge^C B}$ if and only if it is so in $Q_B$.
\end{enumerate}
\end{enumerate}
\end{theorem}
\begin{proof}
It follows from Propositions \ref{Gabrielquiver} and \ref{simplearrows}.
\end{proof}

From the former result we may deduce the following corollaries.

\begin{corollary}\label{fullquiver} Let $C$ be a coalgebra, $Q_C$ its valued Gabriel quiver and $A$ a subcoalgebra of $C$. Then $Q_{A\wedge A}$ is the full valued subquiver of $Q_C$ whose set of vertices are the simple $C$-comodules contained in $A$.
\end{corollary}

\begin{corollary} Let $C$ be a coalgebra and $A$ a subcoalgebra of $C$. If $C=A\wedge^C A$, then $Q_{C}=Q_{A}$ as valued quivers.
\end{corollary}

\begin{corollary} Let $C$ be a hereditary coalgebra with separable coradical and $A$ a subcoalgebra of $C$. Then the following conditions are equivalent:
\begin{enumerate}[$a)$]
\item $A$ is coidempotent, that is, $A\wedge^C A=A$.
\item $A$ is hereditary and $Q_A$ is the full subquiver of $Q_C$ whose vertices are the simple $A$-comodules.
\end{enumerate}
\end{corollary}
\begin{proof}
Let us assume that $A$ is coidempotent. Then $Q_A=Q_{A\wedge A}$ and hence,
 by Corollary \ref{fullquiver}, $Q_A$ is the full quiver of $Q_C$
 whose vertices are the simple $A$-comodules. Since $C$ is hereditary,
by \cite[Corollary 2.7]{formallysmooth}, we may consider
the inclusions $$A_0\subseteq A\hookrightarrow T_{A_0}(A_1/A_0)\subseteq T_{C_0}(C_1/C_0)=C,$$
 where, by $T_D(N)$, we denote the
cotensor coalgebra of a coalgebra $D$ over a $D$-bicomodule $N$.
By \cite[Theorem 2.6]{formallysmooth}, $(A_0)^\infty\cong T_{A_0}((A_0\wedge^C A_0)/A_0)$, where
$(A_0)^\infty=\cup_{n\geq 0} (A_0)^n$ with $(A_0)^0=A_0$ and $(A_0)^n=(A_0)^{n-1}\wedge^C A_0$
for $n>0$. Now, since $A$ is coidempotent, $(A_0)^n=A_n$ the $n$th piece of the coradical filtration
of $A$ for any $n\geq 0$ and then $A\cong (A_0)^\infty \cong T_{A_0}(A_1/A_0)$.
Therefore $A$ is hereditary.

Conversely, by Corollary \ref{fullquiver}, $Q_A=Q_{A\wedge A}$. Then
$A_0=(A\wedge^CA)_0$ and $A_1=(A\wedge^C A)_1$. By \cite[Corollary 2.7]{formallysmooth},
$$A\subseteq (A\wedge^C A)\hookrightarrow T_{A_0}(A_1/A_0).$$
Since $A$ is hereditary, $T_{A_0}(A_1/A_0)\cong A$ and then $A\wedge^CA=A$.

\end{proof}

\section{Semiprime coalgebras}

Let us introduce semiprime coalgebras as a coalgebraic analog to semiprime algebras.
 We will say that $C$ is semiprime if, whether $C=A\wedge A$, for some subcoalgebra
 $A$ of $C$, it happens that $C=A$. Obviously, if $C$ is prime, then $C$ is semiprime.
Also, it is easy to see that $C$ is semiprime if and only if $C^*$ is semiprime.

\begin{lemma}
$C$ is semiprime if and only if it satisfies that, in case $C=\wedge^n A$ for some $n\in \mathbb{N}$, then $C=A$.
\end{lemma}
\begin{proof}
It is enough to prove the necessity of that property. But note that, if $n$ is even, $C=(\wedge^{n/2} A) \wedge (\wedge^{n/2}A)$ and, by hypothesis, $C=\wedge^{n/2}A$. Similarly, if $n$ is odd, $C=\wedge^n A=\wedge^{n+1} A$, and then $C=\wedge^{(n+1)/2}A$. So we may reduce the parameter until $C=A\wedge A$, and then $C=A$.
\end{proof}

\begin{corollary}
Let $C$ be a semiprime coalgebra with finite coradical filtration then $C$ is semisimple.
\end{corollary}
\begin{proof}
By hypothesis, $C=\wedge^n \Soc C$ for some $n\in \mathbb{N}$. Thus $C=\Soc C$ and we are done.
\end{proof}

Observe that semiprime coalgebras are not necessarily indecomposable,
unlike it happens with prime coalgebras. For instance,
a direct sum of simple coalgebras is semiprime non-prime.
The following proposition states that the class of semiprime
coalgebras is closed under direct sums. Hence, from the point
of view of the representation theory,  semiprime coalgebras
seem to be a more appropriate class than merely prime coalgebras.

\begin{proposition}
The direct sum of semiprime coalgebras is semiprime as well.
\end{proposition}
\begin{proof}
Let $C=\oplus_{i\in I}D_i$, where $D_i$ is semiprime for any $i\in I$. Let $A$ a subcoalgebra of $C$ such that $C=A\wedge^C A$. Then $A=\oplus_{i\in I}A_i$, where $A_i$ is a subcoalgebra of $D_i$ for any $i\in I$. Moreover, $A_i\neq 0$  and $D_i=A_i\wedge^{D_i}A_i$ for any $i\in I$. Thus $D_i=A_i$ for any $i\in I$ and hence $A=C$.
\end{proof}

\subsection{Localization in semiprime coalgebras}

The theory of localization has turned out to be a nice tool for
developing the representation theory of coalgebras, see
for instance \cite{jmn2}, \cite{navarro} or \cite{simson07} for some
notions and results. Therefore one might ask oneself about apply
that theory to our problems concerning semiprime coalgebras. Let us
first to characterize semiprime coalgebras by means of its local
structure following the spirit of \cite[Theorem 4.3]{jmnr}.
For any $i\in I_C$,
we denote by $e_i$ the primitive idempotent in $C^*$
which corresponds to the simple $C$-comodule $S_i$.

\begin{theorem}\label{semiprime}
Let $C$ be an arbitrary coalgebra. $C$ is semiprime
if and only if the  coalgebra $eCe$ is semiprime
for any idempotent $e\in C^*$ which is the sum of two primitive orthogonal idempotents.
\end{theorem}
\begin{proof}
Let us assume that $C=A\wedge^C A$ and $C\neq A$. Let $x\in C\backslash A$.
 Then (the right $C$-comodule generated by $x$) $\langle x \rangle$
 is finite dimensional. Assume that $E_1\oplus \cdots \oplus E_n$
 is its injective envelope, where $E_i$ are indecomposable injective
 right $C$-comodules. Hence, there exists $i\in\{1,\ldots , n\}$ such
that $xe_i\neq 0$ and $xe_i\notin A$. Indeed, if $x=\sum_{i=1}^n\lambda_i x_i$,
 where $x_i\in E_i$ and $\lambda_i$ are nonzero scalars, then $x_i=xe_i$ for any
 $i$ and, since $x\notin A$, there exists certain $i$ such that $x_i\notin A$.
 Applying the left version of the above procedure to such element $x_i$
we may find a nonzero element $y\in C$ such that $e_jye_i=y$ and $y\notin A$.
 Let us consider $e=e_i+e_j$, since $eCe=eAe\wedge eAe$ and also
$y\in eCe\backslash eAe$, $eCe$ is not semiprime. The converse may be proven
 as in \cite[Theorem 4.2]{jmnr}.
\end{proof}

By a similar reasoning we may mend, partially, a gap in the proof of
 \cite[Theorem 4.3]{jmnr}.
There, the authors attend to the localization techniques
in the context of prime subcoalgebras of path coalgebras. Nevertheless,
in its proof,  it is only proven  that $C=A\wedge B$ and $x\in C\backslash A$ yields $x\in B$. Then, obviously,
the proof is not completed. In the following proposition we prove a version of
\cite[Theorem 4.3]{jmnr} for arbitrary coalgebras. A coalgebra is said to be socle-finite
\cite{simson07} whether its socle is finite dimensional, equivalently, whether the set of
pairwise non-isomorphic simple comodules is finite.

\begin{proposition}\label{prime}
An arbitrary coalgebra is
 prime if and only any socle-finite ``localized'' coalgebra is prime.
\end{proposition}
\begin{proof}
Let $C$ be an arbitrary coalgebra such that any ``localized'' coalgebra $eCe$ is prime.
Assume that $C=A\wedge^C B$ with $C\neq A$ and $C\neq B$. Then there exist nonzero $x\in C/A$
and $y\in C/B$. Proceeding as in Theorem \ref{semiprime}, there exist $i,j,k,l\in I_C$ such that
$e_ixe_j=x$ and $e_kye_l=y$. Let $e=e_i+e_j+e_k+e_l$. Then $eCe=eAe\wedge eBe$ with $eCe\neq eAe$
and $eCe\neq eBe$. Hence $eCe$ is not prime, a contradiction.
\end{proof}

Then, following the former proof, we should only consider ``localized''
 coalgebras with at most four non-isomorphic
simple comodules.

\subsection{Shape of the valued Gabriel quiver}

\begin{theorem}\label{main} The following assertions hold:
\begin{enumerate}[$i)$]
\item Let $C$ be a semiprime coalgebra, then each connected component of the right (left) valued Gabriel quiver of $C$ is strongly connected.
\item Let $C$ be a hereditary coalgebra whose right (left) valued Gabriel quiver is strongly connected, then $C$ is prime.
\end{enumerate}
As a consequence, any hereditary semiprime coalgebra is the direct sum of prime coalgebras
\end{theorem}
\begin{proof}
\begin{enumerate}[$i)$]
\item Let $S_i$ and $S_j$ be two simple right $C$-comodules in the same connected component of $Q_C$. Then there exists a minimal (non-oriented) path
$$\xymatrix{S_1=S_i \ar@{-}[r] & S_2  \ar@{-}[r] & \cdots \ar@{-}[r]  & S_{n-1} \ar@{-}[r]  & S_j=S_n}$$ in $Q_C$, where $\xymatrix{S_i \ar@{-}[r] & S_{i+1}}$ represents $\xymatrix{S_i \ar[r] & S_{i+1}}$ or $\xymatrix{S_i  & S_{i+1} \ar[l]}$. Let us prove that, for any $i=1,\ldots n-1$, there exists a path in $Q_C$ from $S_i$ to $S_{i+1}$.  Indeed, if there is an arrow $\xymatrix{S_i \ar[r] & S_{i+1}}$, we are done. If not, there exists an arrow $\xymatrix{S_i  & S_{i+1} \ar[l]}$ and then $S_{i+1}$ is a predecessor of $S_i$, i.e., $\Rad_C(E_i,E_{i+1})=\Hom_C(E_i,E_{i+1})\neq 0$.

We consider the idempotent $e\in C^*$ associated to the simple comodules $S_i$
and $S_{i+1}$, and its ``localized'' coalgebra $D=eCe$.
Then $\Hom_{D}(\overline{E}_i,\overline{E}_{i+1})\cong \Hom_C(E_i,E_{i+1})\neq 0$, where $\{\overline{E}_i\}_{i\in I_{D}}$ is a complete set of pairwise non-isomorphic indecomposable injective $D$-comodules. If
$\Hom_{D}(\overline{E}_{i+1},\overline{E}_i)=0$, there is no arrow in $Q_{D}$ from $S_i$ to $S_{i+1}$. So $Q_{D}$ is a subquiver of $$\xymatrix@R=50pt{ S_{i+1} \ar[r] \ar@(ur,ul)[] & S_i \ar@(ur,ul)[]}.$$ It is not difficult to see that $D=A\wedge A$, where $A=e_iDe_i\oplus e_{i+1}De_{i+1}$ (see the proof of
\cite[Theorem 6.2]{GTN}), and therefore $D$ is not semiprime and contradicts Theorem \ref{semiprime}. Hence $0\neq \Hom_{D}(\overline{E}_{i+1},\overline{E}_i)\cong \Hom_C(E_{i+1},E_{i})$ and $S_i$ is a predecessor of $S_{i+1}$. By \cite[Theorem 1.9]{navarro}, there exists a path in $Q_C$ from $S_i$ to $S_{i+1}$.
\item Let us first suppose that $C$ is colocal. Then $C$ is the unique indecomposable injective right (or left) $C$-comodule. Furthermore, any non-zero morphism of right (or left) $C$-comodules $f:C\rightarrow C$ is surjective, since $C$ is hereditary. Assume that $C=A\wedge^C B$, where $A$ and $B$ are proper subcoalgebras of $C$.
     Hence $0\neq C/A\cong C^{(\delta)}$ and $0\neq C/B\cong C^{(\lambda)}$ and therefore there exist a non-zero morphism of right $C$-comodules $f:C\rightarrow C$ such that $f(A)=0$ and a non-zero morphism of left $C$-comodules $g:C\rightarrow C$ such that $g(B)=0$. Now, since $\Delta_C(x)\in A\otimes C+C\otimes B$,
    $$(\Delta\circ f)(x)=(f\circ I)\Delta(x)\in C\otimes B$$ and then, for any $x\in \Ima f=C$, $\Delta(x)\in C\otimes B$. Applying the same reasoning to $g$,
$$(\Delta\circ g)(x)=(I\circ g)\Delta(x)=0$$ for any $x\in C$, and we get a contradiction.

Let now $C$ be a hereditary coalgebra where $Q_C$ is strongly connected.
 If $C$ is not prime, $C=A\wedge^C B$ with $C\neq A$ and $C\neq B$.
 We denote by $E_i^C$, $E_i^A$ and $E_i^B$ the indecomposable injective $C$-comodule, $A$-comodule and $B$-comodule, respectively, with simple socle $S_i$.
Since $C\neq A$, there exists certain indecomposable injective right
 $C$-comodule $E^C_i$ such that $E_i^A\subsetneqq E_i^C$.
Then it may be found a morphism between indecomposable injective
 right $C$-comodules $f:E_i^C\rightarrow E_j^C$ such that $f(E_i^A)=0$.
 Analogously, there exists $g:E_k^C\rightarrow E_t^C$ morphism
 of right $C$-comodules between indecomposable injective
 $C$-comodules such that $g(E_k^B)=0$.

By hypothesis, $Q_C$ is strongly connected so there exists a path from $S_j$ to $S_k$ and a path from $S_t$ to $S_i$. Then, by \cite{navarro}, $\Hom_C(E_t^C,E_i^C)\neq 0$ and $\Hom_C(E_j^C,E_k^C)\neq 0$. So let $h:E_j^C\rightarrow E_k^C$ and $p:E_t^C\rightarrow E_i^C$ be non-zero surjective morphisms of right $C$-comodules. We may consider the surjective morphism $T=pghf\in \End_C(E_i^C)$. It verifies that
$$T(E^A_i)=pgh(f(E^A_i))=0 \qquad \text{and} \qquad T(E^B_i)\subseteq pgh(E^B_j)\subseteq pg(E_k^B)=0$$
Hence, if $e_i$ is the primitive idempotent associated to $E_i^C$, $0\neq e_iT\in \End_{e_iCe_i}(e_iCe_i)$ with $e_iT(e_iAe_i)=e_iT(e_iBe_i)=0$. Thus $e_iAe_i$ and $e_iBe_i$ are proper subcoalgebras of $e_iCe_i$ and $e_iCe_i=e_iAe_i\wedge^{e_iCe_i}e_iBe_i$. Therefore, $e_iCe_i$ is colocal, hereditary and non-prime, a contradiction.
\end{enumerate}
In order to prove the consequence, if $C$ is a hereditary semiprime coalgebra, by $i)$, each connected component of $Q_C$ is strongly connected. By \cite{simson07}, each connected component corresponds to the valued Gabriel quiver of an indecomposable direct sum of $C$. Then these summands are hereditary with strongly connected valued Gabriel quiver. By $ii)$, they are prime.
\end{proof}

We recall from \cite{nttstrictly} that
a right $C$-comodule $M$ is said to be strictly
 quasi-finite
  if every quotient of $M$ is quasi-finite.

\begin{corollary}[Weak Eisenbud-Griffith Theorem for coalgebras]
\label{eisenbud}
Any hereditary semiprime strictly quasifinite coalgebra is serial.
\end{corollary}
\begin{proof}
Let $C$ be a coalgebra as stated above. By the former theorem,
$C$ is a direct sum of prime coalgebras.
Then each direct summand is prime hereditary
 and strictly quasi-finite.
By \cite[Theorem 5.2]{GTN},
$C$ is a direct sum of serial coalgebras.
Then $C$ is serial.
\end{proof}

\begin{corollary} Let $C$ be a indecomposable pointed coalgebra. The Gabriel quiver of $C$ is strongly connected
if and only if $C$ is an admissible subcoalgebra of a prime path coalgebra.
\end{corollary}
\begin{proof}
By \cite{woodcock}, any pointed coalgebra $C$ is isomorphic to
 an admissible subcoalgebra of a path coalgebra $kQ$, i.e., $C$ contains
the set of all the arrows and all the vertices. In particular, $Q_C=Q$. Therefore,
by Theorem \ref{main}, $kQ$ is prime if and only $Q=Q_C$ is strongly connected.
\end{proof}

\begin{remark}\label{counter} Let us remark that, unlike it happens with hereditary coalgebras, there exist
indecomposable semiprime non-prime coalgebras. For instance, let us consider the quiver $Q$ formed
by a single vertex $x$ and two different loops $a$ and $b$ starting and ending at $x$. Now, let
$A$ be the coalgebra generated by $\{x, \{a^n\}_{n>0}\}$ and $B$ be the
 coalgebra generated by $\{x, \{b^n\}_{n>0}\}$ as vector spaces. Both are isomorphic to
the polynomial coalgebra and then prime. Let $R=\{x,\{a^n\}_{n>0},\{b^m\}_{m>0}\}$.
This is a non-prime coalgebra, since $R=A\wedge^R B$.
 Nevertheless, it is semiprime. Indeed, if $R=D\wedge^RD$,
then $\Delta(d)\in D\otimes R+R\otimes D$. Now, for any $n\in \mathbb{N}$,
$\Delta(a^{2n})=a^n\otimes a^n+\text{other terms}$. Then $a^n\in D$. Analogously,
$b^n\in D$. Then $D=R$.

Actually, in this way, we may find an infinite family of non-prime semiprime coalgebras.
Simply, consider $T_n$ the subcoalgebra of $kQ$ generated by all
 paths of length lower or equal than $n$ and
the set $\{a^m,b^m\}_{m>n}$.

Observe as well that this example gives
 a negative answer to the conjecture stated in \cite{nek}, since $R$ is colocal, cocommutative
  and infinite dimensional, and, in despite of this, it is not prime.

\end{remark}

\section{Applications to representation theory}\label{secrep}

Let us finish the paper dealing with its main aim, i.e.,
study the representation theory of a semiprime coalgebra.
Firstly, we show that the assumption of being a basic coalgebra is not a very restrictive condition
since primeness and semiprimeness are preserved under Morita-Takeuchi equivalence. We remind the reader
that, following \cite{chin} or \cite{simson1}, any coalgebra is Morita-Takeuchi
 equivalent to a basic one.

\begin{lemma} Let $C$ and $D$ be two coalgebras which are Morita-Takeuchi equivalent. Then
$C$ is prime (semiprime) if and only if $D$ is prime (semiprime).
\end{lemma}
\begin{proof}
Suppose first that $D$ and $C$ are socle-finite coalgebras. Then,
by \cite{lin}, $C$ and $D$ are strongly equivalent and then $C^*$ and $D^*$
are Morita equivalent. Thus $C^*$ is prime (semiprime) if and only if $D^*$ is prime
(semiprime) and the statement
is proved. For an arbitrary dimension of the socle, let $e$ be an idempotent in $C^*$ corresponding to a socle-finite
injective $C$-comodule $E$. Let $E'$ the injective $D$-comodule image under
 the Morita-Takeuchi
equivalence and $e'\in D^*$ an idempotent associated to $E'$.
Then $eCe$ and $e'De'$ are socle-finite and Morita-Takeuchi equivalent.
By the arguments given above, $eCe$ is prime (semiprime) if and only if $e'De'$ is
prime (semiprime). Hence, by Theorem \ref{semiprime} and Proposition \ref{prime},
$C$ is prime (semiprime) if and only if $D$ is prime (semiprime).
\end{proof}

It is clear that simple subcoalgebras of $C$ are prime and,
 because of this, it is commonly said that prime coalgebras
 are a generalization of simple coalgebras.
Nevertheless, a non-simple prime (resp. semiprime) coalgebra is quite
 far from being a simple (resp. semisimple) one.
 For instance, they are infinite dimensional or, moreover,
 have an infinite coradical filtration.
Let us now show that the category of comodules of a semiprime coalgebra lacks
certain finiteness conditions. A coalgebra is said to be left (right) semiperfect if
all indecomposable injective right (left) $C$-comodules are finite dimensional.

\begin{proposition}\label{semiperfect}
Any semiprime left (right) semiperfect coalgebra is semisimple.
\end{proposition}
\begin{proof}
Assume, contrary to our thesis, that $C$ is not semisimple. Then there exists an
non-simple indecomposable injective right $C$-comodule $E_i$. Therefore,
$E_i/S_i\neq 0$, and there exists certain simple $C$-comodule $S_j\subseteq E_i/S_i$.
We consider the injective right $C$-comodule $E=E_i\oplus E_j$, and the ``localized''
coalgebra $D$ given by this injective comodule. Then $D$ is a socle-finite
left semiperfect, i.e., it is finite dimensional. Since $D$ is semiprime, $D$ is semisimple.
Nevertheless, the quotient functor $T:\M^C\rightarrow \M^D$ is exact, and then
$S_j\subseteq T(E_i)/S_i$, where $T(E_i)$ is the indecomposable injective $D$-comodule whose
socle is $S_i$, see \cite{navarro}. Thus $T(E_i)$ is not simple, and we get a contradiction.
\end{proof}

We may go further  and prove that any Hom-computable semiprime coalgebra is
semisimple. We remind the reader that,
 following \cite{simson07b}, a coalgebra is said
 to be right Hom-computable if the vector space $\Hom_C(E,E')$
 is finite dimensional for any pair of indecomposable
 injective right $C$-comodules $E$ and $E'$.
Clearly, any left semiperfect coalgebra is right
 Hom-computable.

\begin{proposition}
Any semiprime right (left) Hom-computable coalgebra is semisimple.
\end{proposition}
\begin{proof}
Following the proof of Proposition \ref{semiperfect}, it is enough to prove that
any ``localized'' coalgebra relative to an injective $E=E_i\oplus E_j$, where $E_i$ and $E_j$
are indecomposable, is
semisimple. So let $C$ be semiprime right Hom-computable coalgebra and
$E=E_i\oplus E_j$ be an injective right $C$-comodule. Then $(eCe)^*\cong
\Hom_{eCe}(eCe)\cong \Hom_C(E,E) $ is finite dimensional, where $E=Ce$. Then $eCe$ is semisimple.
\end{proof}

It's well-known that the class of all finite dimensional algebras
 over an algebraically field is
divided into two disjoint classes. This is called the tame-wild dichotomy, see for example
\cite{simsonbluebook}, \cite{simsonskowron2} or \cite{simsonskowron3}. On the one side, the
class of tame algebras whose  finitely generated indecomposable modules of a fixed dimension
can be recovered by a finite number of one-parameter families. On the other side, the class
of wild algebras which verify that a classification of its
 finitely generated indecomposable modules yields such classification of any finite dimensional
algebra. Hence, a complete description is feasible for tame algebras only.

When working on coalgebras, we may define tameness and
 wildness following the same spirit of above
\cite{simson1} \cite{simson2}. Nevertheless, the tame-wild dichotomy is still
an open problem. These definitions are slightly changed in order to treat the category
of finitely cogenerated comodules, see
 \cite{simson08} and \cite{simson09}. This is quite natural since indecomposable injective
comodules are commonly infinite dimensional and then they are not considered following
 the classical notions. We denote by  $\M^C_{fc}$ the category of finitely copresented
right $C$-comodules.

 Throughout $k$ will be an algebraically closed field. Then,
 the coalgebra
$C$ is Morita-Takeuchi equivalent to an admissible
 subcoalgebra of the path coalgebra of its Gabriel quiver, see \cite{woodcock}.
 Hence we shall assume
that $C\subseteq kQ$ for certain quiver $Q$ and $kQ_1\subseteq C$.
We recall from \cite{simson08} and \cite{simson09} that,
 for any finitely copresented $C$-comodule $N$ with a minimal injective copresentation
$$\xymatrix{0 \ar[r] & N \ar[r] & E_0 \ar[r]^-{g} & E_1,}$$ the \emph{coordinate vector} of $N$ is the bipartite vector $\mathbf{cdn}(N)=(v_0 | v_1)\in K_0(C)\times K_0(C)=\mathbb{Z}^{(I_C)}\times \mathbb{Z}^{(I_C)}$, where $v_0=\mathbf{lgth}(\Soc E_0)$ and $v_1=\mathbf{lgth}(\Soc E_1)$. We also recall that a $C$-$k[t]_h$-bicomodule $L$ is said to be finitely copresented if there is a $C$-$k[t]_h$-bicomodule exact sequence
$$\xymatrix{0 \ar[r] & L \ar[r] &k[t]_h\otimes E'  \ar[r]^-{g} &  k[t]_h\otimes E''}$$
such that $E'$ and $E''$ are socle-finite injective
 $C$-comodules. Here $h(t)$ is a non-zero polynomial in $k[t]$ and $k[t]_h=k[t,h(t)^{-1}]$
is a rational algebra.

 A $k$-coalgebra $C$
 is said to be of \emph{fc-tame comodule type} (resp. \emph{$k$-tame comodule type}) (cf.
 \cite{simson1} and \cite{simson08})
if for every bipartite vector $v=(v_0|v_1)\in K_0(C)\times K_0(C)$ (resp. length vector
$v\in K_0(C)$)
there exist $k[t]_h$-$C$-bimodules $L^{(1)},\ldots , L^{(r_v)}$,
which are finitely copresented (resp. finitely generated free $k[t]_h$-modules), such that all but finitely
many indecomposable right $C$-comodules $N$ in $\M^C_{fc}$ (resp. $\M^C_{f}$)
with $\mathbf{cdn} M=v$ (resp. $\underline{\rm{length}} M=v$) are of the form
 $M\cong k^1_\lambda\otimes_{k[t]}
L^{(s)}$, where $s\leq r_v$, $k^1_\lambda=k[t]/(t-\lambda)$ and
$\lambda \in k$.

We recall from \cite{simson1} that a finite
dimensional coalgebra $C$ is $k$-tame if and only if its dual algebra $C^*$ is $k$-tame.

A $k$-coalgebra $C$ is of
\emph{fc-wild comodule type} (resp. \emph{$k$-wild comodule type}) if
the category $\M^C_{fc}$ (resp. $\M^C_{f}$) is of wild type, that is,
if there exists an exact
and faithful $k$-linear functor $F: \M^f_{kQ} \rightarrow \M^C_{fc}$ (resp.
$F: \M^f_{kQ} \rightarrow \M^C_{f}$)
that respects isomorphism classes and carries indecomposable right
$kQ$-modules to indecomposable right $C$-comodules, where $Q$ is the quiver
$ \xymatrix{ \circ \ar@<0.6ex>[r]
\ar@<-0.4ex>[r] \ar@<0.1ex>[r] & \circ }$ and $\M^f_{kQ}$ is the category of finitely
generated right $kQ$-modules.

It is easy to see that if $C$ contains a $k$-wild subcoalgebra, then $C$ is $k$-wild.

\begin{lemma}\cite[Proposition 5.1(b)]{simson09}\label{fc-tame}
Let $C$ be an arbitrary coalgebra.
$C$ is fc-tame if and only if any socle-finite ``localized''
coalgebra of $C$ is fc-tame.
\end{lemma}

In general, both definitions of tameness and wildness are unrelated. Nevertheless, as pointed out in \cite{simson07b}, when
the category $\M^C_f$ is included in $\M^C_{fc}$,
  or equivalently,
  all the simple $C$-comodules are finitely copresented, or equivalently,
if $E/\Soc E$ is quasi-finite for any
indecomposable injective $C$-comodule $E$, fc-tameness implies
 $k$-tameness and
$k$-wildness implies fc-wildness. For instance, this holds if $C$ is right
 strictly quasi-finite \cite{nttstrictly}.
The reader should observe that if a simple $C$-comodule is not
finitely copresented, there is an infinite number of arrows ending at a
vertex of $Q$. Thus
$C$ is $k$-wild, since it contains the finite dimensional coalgebra $kQ_1$, where $Q$
is one of the following finite quivers (cf. \cite{simson1}, \cite{simson2}):
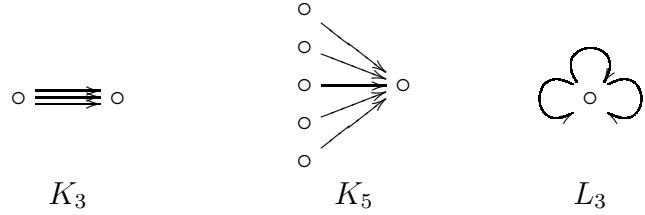
$$\begin{array}[r]{ccc}
\xymatrix{
& \\ \circ \ar@<0.6ex>[r]
\ar@<-0.4ex>[r] \ar@<0.1ex>[r] & \circ   } &
\qquad \qquad \xymatrix@R=2pt{ \circ \ar[rdd] &  \\
\circ \ar[rd] & \\
\circ \ar[r] & \circ \\
\circ \ar[ru] &  \\
\circ \ar[ruu] &  \\
 } &
\qquad \qquad \xymatrix{ \\ \circ \ar@(ul,dl)[] \ar@(ul,ur)[]  \ar@(ur,dr)[]  } \\
K_3 &\qquad \qquad K_5 & \qquad \qquad L_3 \end{array}$$

Therefore it seems reasonable to assume that $Q$ is locally finite, i.e., any vertex
of $Q$ has finite number of arrows starting or ending at it.

\begin{proposition} Let $C$ be a hereditary semiprime $k$-tame coalgebra.
 Then $C$ is serial.
\end{proposition}
\begin{proof}
By \cite{justus}, $Q_C$ must be a disjoint union of extended Dynkin quivers.
Furthermore, by Theorem \ref{main},
each connected component of $Q_C$ is strongly connected. Then, each
 connected component of $Q_C$
must be $\mathbb{A}_n$ for some $n>0$. Then, by \cite[Theorem 2.5]{GTN}, $C$ is serial.
\end{proof}

Nevertheless, if we drop the hereditariness of the coalgebra, the above statement fails.

\begin{example}
Let $C$ be the subcoalgebra of the path coalgebra $kQ$,
where $Q$ is the quiver formed by a single vertex $x$
and two loops $a$ and $b$ starting and ending at $x$,
generated by $\{x,\{a^n\}_{n>0},\{b^m\}_{m>0}\}$.
By Remark \ref{counter}, $C$ is semiprime. Also, $C$ is a string coalgebra in the sense
of \cite{simson2}, and hence it is of $k$-tame comodule type. Nevertheless, it is clear that
$C$ is not serial, see \cite{GTN}.
\end{example}

The following result will be used in the proof of Theorem \ref{biserial}.
By a monomial coalgebra we mean
a subcoalgebra of a path coalgebra generated by paths as vector space.
We also remind from \cite{simson2} that an admissible subcoalgebra $C$ of a path
coalgebra $kQ$ is said to be string if satisfies the following properties:
\begin{enumerate}[$a)$]
\item each vertex of $Q$ is the source
of at most two arrows and the sink of at most two
arrows.
\item $C$ is a monomial coalgebra.
\item given an arrow $\beta:i\rightarrow j$ in $Q$, there is at most one arrow
$\alpha:j\rightarrow k$ in $Q$, and at most one arrow $\gamma:l\rightarrow i$
in $Q$ such that $\alpha\beta\in C$ and $\beta\gamma\in C$.
\end{enumerate}

\begin{lemma}\label{bicycle}
Let $C$ be a monomial admissible subcoalgebra of $kQ$, where $Q$ is the quiver formed by
one vertex $x$ and two loops $a$ and $b$.
 If $C$ is semiprime and $k$-tame, then  it is the coalgebra generated by one of the
following sets:
\begin{enumerate}[$a)$]
\item the paths $\{x, \{a^n\}_{n>0},\{b^m\}_{m>0}\}$, or
\item the paths
$\{x,\{b(ab)^t\}_{t\geq 0}, \{a(ba)^s\}_{s\geq 0}, \{(ab)^n\}_{n>0},\{(ba)^m\}_{m>0}\}$.
\end{enumerate}
As a consequence, $C$ is a string coalgebra.
\end{lemma}
\begin{proof}
Since $C$ is $k$-tame, by the Weak Tame-Wild Dichotomy proved in \cite{simson1}, $C$ is not $k$-wild. Let us suppose that $a^2\in C$ and $ba\in C$.
 Then $C$ contains the subcoalgebra
generated by $\{x,a,b,a^2,ba\}$, which is dual
to the finite dimensional algebra $k\langle a,b\rangle$ with relations
$a^3=b^2=ba^2=ab=0$. By \cite{ringel75}, this algebra is $k$-wild and then $C$ is so.

Let us suppose that $a^2\in C$ and $ba\notin C$.
If $ab\in C$, we may repeat the above arguments and get a
contradiction.
Now, if the powers of $a$ are bounded, i.e., there exists certain $t\geq 2$
such that $a^t\in C$ and $a^{t+1}\notin C$, consider
$A=C\cap \langle \{a^sb^m \text{ such that } s<t, m\geq 0\}\rangle$. Clearly,
$A\wedge^C A=C$ and $A\neq C$, since $a^t\notin A$. Then, $C$ is not semiprime. Similarly, the powers of $b$ are not
bounded. Then $C$ is the coalgebra of $a)$.

Let us suppose that $a^2\notin C$ and $ba \in C$. If $ab\notin C$, then $C$ must
be generated by paths $b^ma$ for $m\geq 0$. Then, $B\wedge^C B=C$, where
$B=\langle \{b^m\}_{m\geq 0}\}$, and $C$ is not semiprime. Then $ab\in C$ and,
by similar arguments as above, $b^2\notin C$. Since $C$ has infinite coradical filtration,
 i.e., the length of the paths in $C$ is unbounded,
$C$ is the coalgebra of $b)$.

Finally, if $a^2\notin C$ and $ba \notin C$, the paths in $C$ has the form $ab^m$
 for $m\geq 0$ and therefore $C=B\wedge B$, where $B=\langle \{b^m\}_{m\geq 0}\}$,
 and it is not semiprime.

It is easy to see that the coalgebras of $a)$ and $b)$ are semiprime and string.
\end{proof}

\begin{theorem} \label{biserial}
Let $Q$ be a locally finite quiver and $C$ a monomial admissible subcoalgebra
of $kQ$. If $C$ is a semiprime fc-tame coalgebra,
then $C$ is a string coalgebra.
\end{theorem}

\begin{proof}
Since $Q$ is locally finite,
the number of arrows ending at any vertex of $Q$ is finite, so each simple $C$-comodule
is finitely copresented. Then
$\M^C_{f}$ is contained in $\M^C_{fc}$ and,
by the Weak Tame-Wild Dichotomy proved in \cite{simson1}, $C$ is not $k$-wild.

Let us prove that in the quiver $Q$ the vertices are the source of at most two arrows.
Suppose, contrary to this, that $x\in Q_0$ is the source of three arrows, and we prove that either
$C$ is $k$-wild, or a ``localized'' coalgebra of $C$ is $k$-wild. If the arrows end
at the same vertex, $Q$ contains one of the following subquivers $\Gamma$:
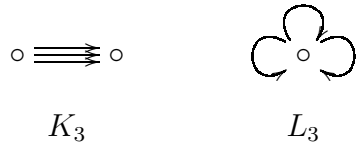
$$\begin{array}[r]{cc}
\xymatrix{
& \\ \circ \ar@<0.6ex>[r]
\ar@<-0.4ex>[r] \ar@<0.1ex>[r] & \circ  } &
\qquad \qquad \xymatrix{ \\ \circ \ar@(ul,dl)[] \ar@(ul,ur)[]  \ar@(ur,dr)[]  } \\
 & \\
K_3 & \qquad \qquad L_3 \end{array}$$
In both cases $(k\Gamma)_1\subset C$ is $k$-wild, see \cite{simson2}.
If the arrows end at two different vertices, $Q$ contains
one of the following subquivers:
$$\begin{array}{ccc}
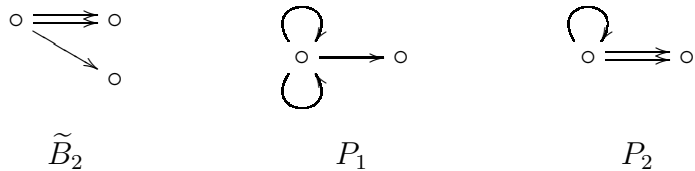
\xymatrix@R=10pt{
& \\ \circ \ar@<0.4ex>[r]
\ar@<-0.2ex>[r] \ar[rd] & \circ  \\
& \circ}
&
\qquad \qquad
\xymatrix{ & \\ \circ \ar@(ul,ur)[] \ar@(dl,dr)[]  \ar[r] & \circ  }
&
\qquad \qquad \xymatrix{ & \\ \circ \ar@(ul,ur)[] \ar@<0.4ex>[r]
\ar@<-0.2ex>[r] & \circ  }
\\ & &\\
\widetilde{B}_2 & \qquad \qquad P_1 & \qquad \qquad P_2 \end{array}$$
$K\widetilde{B}_2\subset C$ is a hereditary coalgebra whose Gabriel quiver is not Dynkin
diagram and then $k$-wild \cite{simson1}. For the other two cases, $(kP_i)_1\subset C$ is
a coradical square complete coalgebra whose separate quiver is not a Dynkin
or Euclidean diagram,
then, by \cite[Theorem 6.5]{justus2}, it is $k$-wild.

Therefore, we may assume that the three arrows end at three different vertices, i.e,
the subquiver has one of the following shapes:
$$\begin{array}{cc}
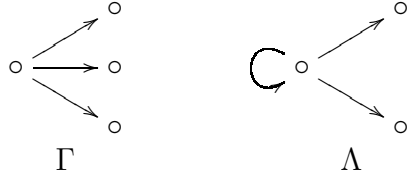
\xymatrix@R=10pt{  & \circ \\ \circ \ar[r]\ar[ru] \ar[rd] & \circ \\ & \circ}
&
\qquad \qquad
\xymatrix@R=10pt{ & \circ \\ \circ \ar@(ul,dl)[] \ar[rd] \ar[ru]& \\ & \circ }
\\
\Gamma & \qquad \qquad \Lambda \end{array}$$
Let us consider the ``localized'' coalgebra $D=eCe$,
where $e$ correspond to the vertices of the above subquivers. Then $D$ is
also generated by paths, semiprime (Theorem \ref{semiprime})
 and fc-tame
\cite[Proposition 5.1]{simson09}. Furthermore, it contains
$\Gamma$ or $\Lambda$ as subquiver, see \cite{jmn2}. Now, by Theorem \ref{main},
 the Gabriel quiver of $D$ is strongly connected. Furthermore, if we make use of
\cite[Theorem 6.5]{justus2} for analyzing the comodule type of the first piece $D_1$ of the
coradical filtration, the only possible quivers for $D$ are the following:
$$\text{Case $\Gamma$:}\begin{array}{ccc}
 &
 
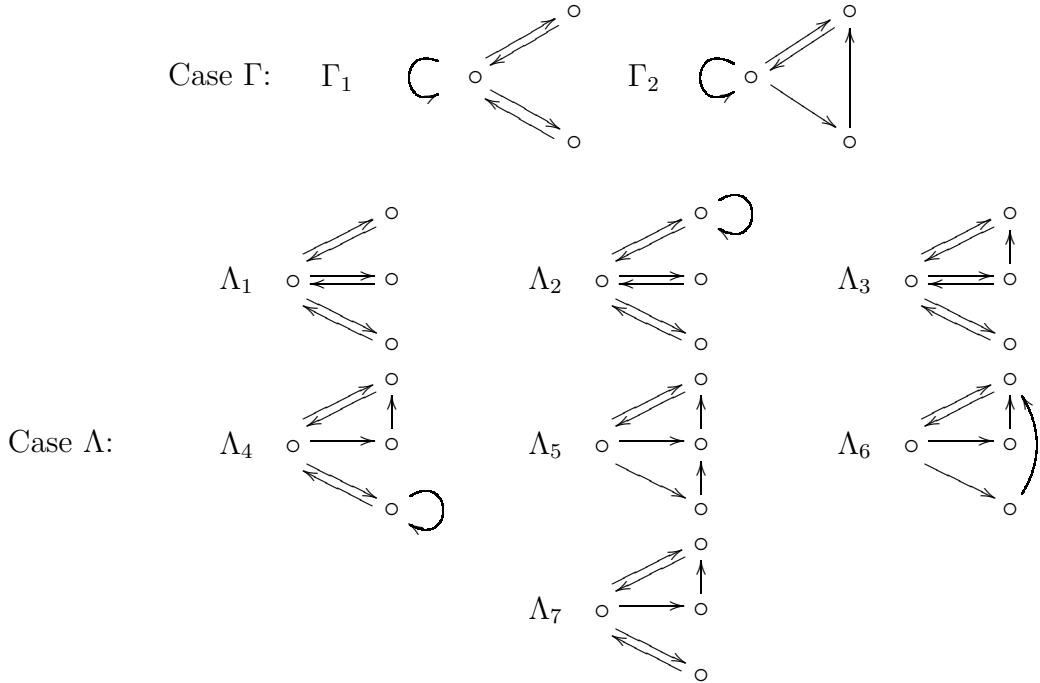
\xymatrix@R=10pt{
& &  \circ \ar@<0.3ex>[ld]\\
\Gamma_1& \quad\circ \ar@(ul,dl)[] \ar@<0.3ex>[ru]  \ar@<0.4ex>[rd]&  \\
& & \circ  \ar@<0.4ex>[lu]}
&
 \xymatrix@R=10pt{
& &  \circ \ar@<0.3ex>[ld]\\
\Gamma_2& \circ \ar@(ul,dl)[] \ar@<0.3ex>[ru]  \ar@<0.4ex>[rd]&  \\
&  &\circ  \ar[uu]} \end{array}$$

$$\text{Case $\Lambda$:}\begin{array}{ccc}
 \xymatrix@R=10pt{
& &  \circ \ar@<0.3ex>[ld] \\
&\Lambda_1 \quad\circ  \ar@<0.3ex>[ru] \ar@<0.3ex>[rd] \ar@<0.3ex>[r]& \circ \ar@<0.3ex>[l]\\
& & \circ  \ar@<0.3ex>[lu] }
&
 \xymatrix@R=10pt{
& &  \circ \ar@<0.3ex>[ld]\ar@(ur,dr)[]\\
&\Lambda_2 \quad \circ  \ar@<0.3ex>[ru] \ar@<0.3ex>[rd] \ar@<0.3ex>[r]& \circ \ar@<0.3ex>[l]\\
& & \circ  \ar@<0.3ex>[lu] }
&
 \xymatrix@R=10pt{
& &  \circ \ar@<0.3ex>[ld] \\
&\Lambda_3 \quad\circ  \ar@<0.3ex>[ru] \ar@<0.3ex>[rd] \ar@<0.3ex>[r]& \circ
\ar[u] \ar@<0.3ex>[l]\\
& & \circ  \ar@<0.3ex>[lu] }
\\
\xymatrix@R=10pt{
& &  \circ \ar@<0.3ex>[ld]\\
 &\Lambda_4 \quad \circ  \ar@<0.3ex>[ru] \ar@<0.3ex>[rd] \ar@<0.3ex>[r]& \circ \ar[u]\\
& & \circ  \ar@<0.3ex>[lu] \ar@(ur,dr)[]}
&
 \xymatrix@R=10pt{
& &  \circ \ar@<0.3ex>[ld]\\
&\Lambda_5 \quad\circ  \ar@<0.3ex>[ru] \ar@<0.3ex>[rd] \ar@<0.3ex>[r]& \circ \ar[u]\\
& & \circ  \ar[u] }
&
 \xymatrix@R=10pt{
& &  \circ \ar@<0.3ex>[ld]\\
&\Lambda_6 \quad\circ  \ar@<0.3ex>[ru] \ar@<0.3ex>[rd] \ar@<0.3ex>[r]& \circ \ar[u]\\
& & \circ  \ar@/_10pt/[uu] }
\\
&  \xymatrix@R=10pt{
& &  \circ \ar@<0.3ex>[ld]\\
 &\Lambda_7 \quad \circ  \ar@<0.3ex>[ru] \ar@<0.3ex>[rd] \ar@<0.3ex>[r]& \circ \ar[u]\\
& & \circ  \ar@<0.3ex>[lu] }  & \end{array}$$
We shall need the following lemma.
\begin{lemma}\label{paths} If there exists an arrow $\alpha$ such that there is no arrow $\beta$
with $\beta\alpha\in C$ ($\alpha\beta\in C$), then $C$ is not semiprime. More generally,
if there is a path $p$ such that $p\beta\notin C$ ($\beta p\notin C$) for any arrow $\beta$,
then $C$ is not semiprime.
\end{lemma}
\begin{proof}
If there is no arrow $\beta$ such that $\beta\alpha\in C$, then, for each path $p\in C$,
or $\alpha\nleq p$, or $p=\alpha q$ with $\alpha \nleq q$. We may consider the subcoalgebra
$$A=C\cap \langle \{\text{paths $p$ in $C$ such that $\alpha\nleq p$}\}\rangle.$$
Since $\alpha\notin A$, $A\neq C$. Nevertheless, $A\wedge^C A=C$. Indeed, for each path $p\in C$,
if $\alpha\nleq p$ then $p\in A$, and, if $p=\alpha q$ with $\alpha\nleq q$, clearly
$\Delta(p)\subseteq A\otimes C+C\otimes A$. Thus $C$ is not semiprime. The proof of the general
case is similar.
\end{proof}
Let us consider that the Gabriel quiver of $D$ is $\Gamma_1$
$$
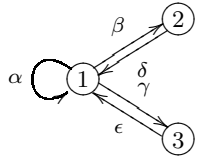
\xymatrix@R=10pt{
&  \point{2} \ar@<0.3ex>[ld]^-{\delta}\\
 \point{1} \ar@(ul,dl)[]_-{\alpha} \ar@<0.3ex>[ru]^-{\beta}  \ar@<0.4ex>[rd]^-{\gamma}
&  \\
 & \point{3}  \ar@<0.4ex>[lu]^-{\epsilon}}$$
By Lemma \ref{paths}, $\delta\beta\in D$. Then
$\delta\beta$ is a cell relative to $\{e_1,e_3\}$ in the sense of \cite{jmn2}. Hence, if
$e=e_1+e_3$, by \cite{jmnr}, the Gabriel quiver of $eDe$ contains $P_1$ as a subquiver, and
then $eDe$ is $k$-wild and not $k$-tame, contradicting Lemma \ref{fc-tame}.
 By a similar reasoning,
me may discard all the other possibilities.
Then each vertex is the source of at most two arrows.
Similarly, it is the sink of at most two arrows.

In the following we will assume that all the quivers involved contain, at most, two arrows
starting/ending from/at any vertex.
Let us now prove that, if $a$ and $b$ are two arrows starting from certain vertex,
 and $c$ is an arrow
ending at this vertex, then $ac\notin C$ or $bc\notin C$. Actually, following Lemma \ref{paths},
 it is proven that $ac\in C$ or $bc\in C$, but not both. In the following list, we describe all
possible cases contradicting our thesis and prove that,
in these cases, the coalgebra, or a localization of it, must be $k$-wild.
By the localization techniques, the reader may assume that
the whole quiver
is reduced to a quiver whose vertices are those showed in each case.
\begin{enumerate}[$a)$]
\item Case $Q=L_2$ and $c=a$ or $c=b$. This is proven
in Lemma \ref{bicycle}.
\item Case $\xymatrix{\circ \ar[r]^-{c} & \circ
\ar@<0.3ex>[r]^-{a} \ar@<-0.3ex>[r]_-{b} & \circ}$ with $ac,bc\in C$. Then $C$ contains
a finite dimensional hereditary coalgebra whose Gabriel quiver is not a Dynkin or Euclidean graph. Then $C$ is $k$-wild.
\item Case $\xymatrix{\circ
\ar@<0.6ex>[r]^-{a,b}  \ar[r] & \circ \ar@<0.6ex>[l]^-{c} }$ with $ac, bc\in C$.
 Then
$C$ contains the dual coalgebra of the $k$-wild radical cube zero algebra of
\cite[Table W(2)]{han}.
\item Case $\begin{array}[r]{c}\xymatrix{\point{1} \ar@(ul,dl)[]_-{a}
\ar[r]^-{b}  \ar[r] & \point{2}\ar@<0.3ex>@{}[l]^-{\phantom{c}} }\end{array}$
with $a^2\in C$ and $ba\in C$. If the powers of $a$ in $C$
are not bounded, then $C$ contains the dual coalgebra of
 the algebra \cite[Table W(6)]{han}.
Therefore $a^t\notin C$ for some $2<t<7$. Now, since the ``localized'' coalgebra $e_1Ce_1$ is
not simple (and then infinite dimensional) semiprime and fc-tame, its Gabriel quiver should
be $L_2$ with two loops named $a$ and $\overline{b}$, where $a^2\in e_1Ce_1$
but $a^7\notin e_1Ce_1$.
By Lemma \ref{bicycle}, this is not possible.
\item Case $\begin{array}[r]{c}\xymatrix{\point{1} \ar@(ul,dl)[]_-{a}
\ar@<0.3ex>[r]^-{b}   & \point{2}\ar@<0.3ex>[l]^-{c} }\end{array}$ with $ac,bc\in C$.
By Lemma \ref{paths}, $bac\in C$ or $a^2c\in C$. If $a^2c\in C$, we are in the dual of
 Case $d)$.
Therefore $bac\in C$ and then $ba\in C$. Hence $a^2\notin C$, since,
otherwise, we are in Case $d)$.
Also, there is no arrow starting at $\isolated{2}$. In such a case,
since $bac$ and $ac$ are cells,
$e_2Ce_2$ has Gabriel quiver $L_3$ or $P_1$. Then $C$ contains
the subcoalgebra generated by $\{1,2,a,b,c,ba,cb,bc,cba,cbc\}$ dual to
the $k$-wild algebra of \cite[Table W(21)]{han}.

\item Case $\begin{array}[r]{c}\xymatrix{\point{1} \ar[r]^-{c} &
\point{2} \ar@(ul,ur)[]^-{a}
\ar[r]^-{b}   & \point{3}}\end{array}$ with $ac,bc\in C$. Let us suppose that $a^nc\in C$
for any $n>0$, then $C$ is $k$-wild since it contains the hereditary coalgebra
of the quiver formed by $c$ and the loop $a$. Hence there exists certain $t>0$
such that $a^tc\notin C$ and $a^{t-1}c\in C$. By Lemma \ref{paths}, $ba^{t-1}c\in C$.
If $t>1$, $ba^2c\in C$ and $ba^2\in C$ so we are in Case $d)$. Then $t=1$, and $bc, bac\in C$.
Now, if there exists another arrow starting at $\isolated{1}$,
which cannot end at $\isolated{2}$, there are three arrows starting at $\isolated{1}$ in
$eCe$ with $e=e_1+e_3$. Then, there exists an arrow $d:\isolated{3}\rightarrow\isolated{1}$.
Observe that, if there exists another arrow $h:\isolated{3}\rightarrow\isolated{1}$, by Lemma
\ref{paths}, $ch,cd\in C$ and we may reduce it to Case $b)$. Furthermore, if there exists a
loop at $\isolated{3}$, the vertex $\isolated{3}$ receives three arrows in $eCe$ with $e=e_1+e_3$.
Therefore, the quiver $Q_C$ is the following
$\begin{array}[r]{c}\xymatrix{\point{1} \ar[r]^-{c} &
\point{2} \ar@(ul,ur)[]^-{a}
\ar[r]^-{b}   & \point{3} \ar@/^10pt/[ll]^-{d}}\end{array}$. Furthermore, by  Lemma \ref{paths}, $cdb\in C$. Then, if $e=e_1+e_2$, $eCe$ is an admissible coalgebra of $\begin{array}[r]{c}\xymatrix{\point{1} \ar@<0.3ex>[r]^-{c} &
\point{2} \ar@(ul,ur)[]^-{a}
\ar@<0.3ex>[l]^-{m}  }\end{array}$
where $m$ is an arrow obtained from the cell $db$, and $mc,ac\in C$. Thus, by Case $e)$, $eCe$ is $k$-wild.

\item Case $\begin{array}[r]{c}\xymatrix@R=10pt@C=30pt{
\point{1}\ar[rd]_-{b} \ar@<0.3ex>[r]^-{a} &
 \point{2} \ar@<0.3ex>[l]^(0.4){c}\\
  &    \point{3} }\end{array}$ with $ac,bc\in C$. Let us determine the shape of the quiver.
Suppose that there is no arrow from $\isolated{3}$ to $\isolated{1}$.
 Since $Q$ is strongly connected, there exists an arrow $d$ from
$\isolated{3}$ to $\isolated{2}$. Moreover, any other arrow
starting from $\isolated{3}$ should
be a loop $\alpha$. If such a loop exists, by Lemma \ref{paths},
either $i)$ there exists a path
 $d\alpha^nbc\in C$ for some $n>0$, or  $ii)$
$\alpha^mb\in C$ and $d\alpha^m\in C$ for any $m>0$.
If $i)$, $eCe$ with $e=e_1+e_2$
belongs to Case $c)$, since $d\alpha^nb$ is a cell.
If $ii)$, as proved in Case $d)$, the coalgebra
is $k$-wild. If there is no loop $\alpha$, the proof is as in $i)$.
 Therefore, $Q$ contains the subquiver
$\begin{array}[r]{c}\xymatrix@R=2pt@C=15pt{
\point{2} \ar@<0.3ex>[r]^-{c}&
\point{1} \ar@<0.3ex>[r]^-{b} \ar@<0.3ex>[l]^-{a}&
 \point{3}\ar@<0.3ex>[l]^-{h}  }\end{array}$. Now, since $ac\in C$,
 a loop at $\isolated{2}$ provokes that
$eCe$ is $k$-wild because of its Gabriel quiver contains $P_1$, where $e=e_1+e_3$.
Then $Q$ can change from the above quiver uniquely if there is a loop at $\isolated{3}$.
 Let us suppose that there is no
such a loop, and the other case may be proven similarly.
 We have $ah\in C$. Indeed, if $ah\notin C$,
there is no path from $\isolated{3}$ to $\isolated{2}$.
 Hence, $Q_{eCe}$ is not strongly connected
with $e=e_2+e_3$.
Now, by Lemma \ref{paths}, $ahb\in C$ and also $ac\in C$.
 Then $Q_{eCe}$ is given  by
$\begin{array}[r]{c}\xymatrix{\point{1} \ar@(ul,dl)[]_-{\alpha}
\ar@<0.3ex>[r]^-{a}   & \point{2}\ar@<0.3ex>[l]^-{c} }\end{array}$
where $e=e_1+e_2$ and $\alpha$ corresponds to the cell $hb$. Hence $\alpha c\in eCe$ and
$ac\in eCe$. By Case $e)$, $eCe$ is $k$-wild.
\item Case $\begin{array}[r]{c}\xymatrix@R=2pt@C=15pt{ &  & \point{3} \\
 \point{1} \ar[r]^-{c}& \point{2} \ar[ru]^-{a} \ar[rd]_-{b}  &   \\
  &  &   \point{4} }\end{array}$ with $ac,bc\in C$.
Let us reduce the possibilities of the graph. Since $Q$ is strongly connected and
 $\xymatrix{\point{2}}$ is the source of at most two arrows,
there exists an arrow $d$ from $\xymatrix{\point{3}}$ (or $\xymatrix{\point{4}}$)
  to $\xymatrix{\point{1}}$.
$$
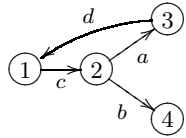
\xymatrix@R=6pt@C=15pt{ &  & \point{3} \ar@/_5pt/[lld]_-{d}\\
 \point{1} \ar[r]_-{c}& \point{2} \ar[ru]_-{a} \ar[rd]_-{b}  &   \\
  &  &   \point{4} }$$
There is no another arrow than $c$ starting at $\isolated{1}$. Indeed, if there is a loop
in $\isolated{1}$, this vertex is the source of three arrows in $eCe$ with $e=e_1+e_3+e_4$, the loop itself and two arrows obtained from
the cells $ac$ and $bc$. Similarly, if there is an
 arrow from $\isolated{1}$ to $\isolated{3}$ or $\isolated{4}$, $Q_{eCe}$ contains
the $k$-wild path coalgebra of $\widetilde{B}_2$.
Finally, if there is an arrow $h$ from $\isolated{1}$ to $\isolated{2}$,
by Lemma \ref{paths}, $ah\in C$ or $bh\in C$. Then, it is dual to Case $b)$.
Also, if there is another arrow $h:\isolated{3}\rightarrow \isolated{1}$,
 by Lemma \ref{paths}, we are in Case $b)$.

Let us assume that there is no arrow from $\isolated{4}$ to $\isolated{1}$. Since $Q$
is strongly connected, there exists
an arrow $h:\isolated{4}\rightarrow \isolated{2}$
or $t:\isolated{4}\rightarrow \isolated{3}$. If there is no such an $h$, there is no another
arrow from $\isolated{4}$ to $\isolated{3}$ since $Q$ cannot contain
$(\widetilde{B}_2)^{op}$. By Lemma \ref{paths}, if there is a loop $\alpha$ at
$\isolated{4}$, either $i)$
$t\alpha^nbc\in C$ for some $n\geq 0$, or $ii)$ $t\alpha^m$ and
$\alpha^mb\in C$ for any $m>0$. If $i)$, $t\alpha^nb$ is a cell relative to
the localizing subcategory associated to $e=e_1+e_2+e_3$. Then $eCe$ contains
the $k$-wild path coalgebra of the quiver of Case $b)$. If $ii)$, as proved in Case
$d)$, $C$ is $k$-wild. If there is not such a loop, the proof is similar to $ii)$.
If there is no such an arrow $t$, we may discuss similarly and, $i)$
may be reduced to Case $f)$ and if $ii)$, the coalgebra is $k$-wild.
In case that there exist both, by Lemma \ref{paths}, $tbc\in C$ or $hbc\in C$ and then
$tb$ or $hb$ becomes cells for $e=e_1+e_2+e_3$.
So $C$ becomes $k$-wild applying the above arguments.
Then $Q$ has a (unique) arrow $m:\isolated{4}\rightarrow \isolated{1}$, and $Q$ contains
$$
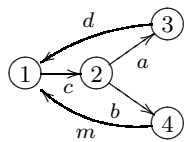
\xymatrix@R=6pt@C=15pt{ &  & \point{3} \ar@/_5pt/[lld]_-{d}\\
 \point{1} \ar[r]_(0.6){c}& \point{2} \ar[ru]_-{a} \ar[rd]_(0.4){b}  &   \\
  &  &   \point{4} \ar@/^8pt/[llu]^-{m} }$$
 Applying the same arguments, it is not difficult to see that $mbc, dac\in C$. Then, in the
most favorable case, for $e=e_1+e_2$, $eCe$ is reduced to Case $c)$.
\end{enumerate}
Duality, we may prove that if $a$ and $b$ are two arrows ending at certain vertex,
 and $c$ is an arrow
starting from this vertex, then $ca\notin C$ or $cb\notin C$. This completes the proof.
\end{proof}


\begin{thebibliography}{99}

\bibitem{Abe} E. Abe, Hopf Algebras, Cambridge University Press, 1977.

\bibitem{abu} J. Y. Abuhlail, Fully coprime comodules and
 fully coprime corings, Appl. Categ. Structures 14 (2006), no. 5-6, 379--409.

\bibitem{abu2} J. Y. Abuhlail,
A Zariski topology for bicomodules and corings,
Appl. Categ. Structures 16 (2008), no. 1-2, 13--28.

\bibitem{brustlehan} Th. Br\"{u}stle and Y. Han, Two-point algebras without loops,
Comm. Algebra (10) 29 (2001), 4683–-4692.

\bibitem{chin} W. Chin and S. Montgomery,  Basic coalgebras, AMS/IP Studies in
Advanced Mathematics 1997, 4, 41--47.

\bibitem{cuadra2} J. Cuadra, Extensions of rational modules,
Int. J. Math. Math. Sci. 69 (2003), 4363--4371.

\bibitem{cuadra} J. Cuadra and J. G{\'o}mez-Torrecillas,
Idempotents and Morita-Takeuchi theory, Comm. Algebra 30 (2002),
2405--2426.



\bibitem{eisenbud} D. Eisenbud and P. Griffith, Serial rings, J.
Algebra 17 (1971), 389--400.

\bibitem{ferrero} M. Ferrero and V. Rodrigues, On prime and semiprime modules and comodules, J. Algebra Appl. 5(5) 2006, 681--694.

\bibitem{gabriel} P. Gabriel, Des categories abeliennes, Bull. Soc. Math. France 90 (1962),
323--448.

\bibitem{han01} Y. Han, Controlled wild algebras, Proc. London Math. Soc. 83 (2) (2001),
279--298.

\bibitem{han} Y. Han, Wild two-point algebras, J. Algebra 247 (2002), 57--77.

\bibitem{nttstrictly} J. G\'{o}mez-Torrecillas, C. N\u ast\u asescu and B. Torrecillas,
Localization in coalgebras. Applications to finiteness conditions,
J. Algebra Appl. 6 (2007), 233-243.

\bibitem{GTN} J. G\'{o}mez-Torrecillas and G. Navarro,
Serial coalgebras and its valued Gabriel quivers, J. Algebra 319 (2008), 5039--5059.

\bibitem{formallysmooth} P. Jara, D. LLena, L. Merino and D. \c{S}tefan,
Hereditary and formally smooth coalgebras,  Algebr. Represent. Theory  8  (2005),  no. 3, 363--374.

\bibitem{jmn2} P. Jara, L. M. Merino and G. Navarro, Localization
in tame and wild coalgebras, J. Pure Appl. Algebra 211 (2007), 342--359.

\bibitem{jmnr2} P. Jara, L. Merino, G. Navarro and J. F. Ru\'{\i}z,
Localization in coalgebras, stable localizations and path
coalgebras,  Comm. Algebra  34 (2006), 2843--2856.

\bibitem{jmnr} P. Jara, L. Merino, G. Navarro and J. F. Ru\'{\i}z,
Prime path coalgebras, Arab. J. Sci. Eng. 33 (2008), Number 2C, 273-283.

\bibitem{justus} J. Kosakowska and D. Simson, Hereditary coalgebras and
representations of species, J. Algebra 293 (2005), 457--505.


\bibitem{justus2} J. Kosakowska and D. Simson, Bipartite
coalgebras and a reduction functor for coradical square complete
coalgebras, Coll. Math. 112 (2008), 89-–129.

\bibitem{lin} I-Peng Lin, B. Morita's theorem for coalgebras,
Comm. Algebra 1 (1974), no. 4, 311--344.

\bibitem{montgomery} S. Montgomery, Hopf Algebras and Their
Actions on Rings, MBS, No. 82, AMS, 1993.


\bibitem{montgomery2} S. Montgomery, Indecomposable coalgebras, simple comodules,
and pointed Hopf algebras,  Proc. Amer. Math. Soc. 123 (1995),
2343--2351.



\bibitem{navarro} G. Navarro, Some remarks on localization in coalgebras,
 Comm. Algebra 36 (2008), 3447--3466.

\bibitem{nek} R. Nekooei and L. Torkzadeh,
Topology on coalgebras, Bull. Iran. Math. Soc. 27(2) 2001, 45--63.

 \bibitem{radford} D. E. Radford,
 On the structure of pointed coalgebras, J. Algebra 77 (1982), 1-14.

\bibitem{ringel75} C. M. Ringel, The representation type of local algebras,
Proceedings of the International Conference on Representations of Algebras
, Carleton Math. Lecture Notes, No. 9, Carleton Univ.,
Ottawa, Ont., 1974.

\bibitem{simsonbluebook} D. Simson, Linear Representations
of Partially Ordered Sets and Vector Space Categories, Algebra Logic
Appl. 4, Gordon \& Breach, 1992.


\bibitem{simson1} D. Simson, Coalgebras, comodules,
pseudocompact algebras and tame comodule type, Colloq. Math. 90
(2001), 101-150.

\bibitem{simson2} D. Simson, Path coalgebras of quivers with
relations and a tame-wild dichotomy problem for coalgebras, Lectures
Notes in Pure and Applied Mathematics 236 (2005), 465-492.


\bibitem{simson06} D. Simson, Irreducible morphisms, the Gabriel-valued quiver and
colocalizations for coalgebras, Intern. J. Math. Sci. 72 (2006),
1--16.

\bibitem{simson07b} D. Simson, Hom-computable coalgebras, a composition factors matrix and the Euler
bilinear form of an Euler coalgebra, J. Algebra 315 (2007), 42--75.

\bibitem{simson07} D. Simson, Localising embeddings of comodule categories with applications
to tame and Euler coalgebras, J. Algebra 312 (2007) 455–-494.

\bibitem{simson08} D. Simson, Tame-wild dichotomy for coalgebras, J. London Math. Soc. (2) 78 (2008) 783–-797.

\bibitem{simson09} D. Simson, Tame comodule type, Roiter bocses, and a geometry context for coalgebras, Ukrainian Mathematical Journal 61 (2009), no. 6, 810--833.


\bibitem{simsonskowron2} D. Simson and A. Skowron'ski,
 Elements of the Representation Theory of Associative Algebras, Vol. 2:
 Tubes and concealed algebras of Euclidean type,
London Mathematical Society Student Texts 71, Cambridge University Press, Cambridge, 2007.

\bibitem{simsonskowron3} D. Simson and A. Skowron'ski,
 Elements of the Representation Theory of Associative Algebras, Vol. 3:
 Representation-infinite Tilted Algebras,
London Mathematical Society Student Texts 72, Cambridge University Press, Cambridge, 2007.

\bibitem{sweedler} M. E. Sweedler, Hoft Algebras, Benjamin,
New york, 1969.

\bibitem{takeuchithesis} M. Takeuchi, Tangent coalgebras and hyperalgebras.
 I, Japan. J. Math. 42 (1974), 1–-143.

 \bibitem{takeuchi} M. Takeuchi, Morita theorems for categories of
comodules, J. Fac. Sci. Uni. Tokyo 24 (1977), 629--644.

\bibitem{WijayantiWisbauer} I. A. Wijayanti and R. Wisbauer,
On coprime modules and comodules, Comm. Algebra 37 (2009), no. 4, 1308--1333.

\bibitem{woodcock}
 D. Woodcock,
 Some categorical remarks on the representation theory of coalgebras,
 Comm. Algebra 25 (1997), 2775--2794.

\end{thebibliography}
\end{document}